\documentclass[12pt,reqno]{article}

\usepackage[usenames]{color}
\usepackage{amssymb}
\usepackage{amsmath}
\usepackage{amsthm}
\usepackage{amsfonts}
\usepackage{amscd}
\usepackage{graphicx}

\usepackage[colorlinks=true,
linkcolor=webgreen,
filecolor=webbrown,
citecolor=webgreen]{hyperref}

\definecolor{webgreen}{rgb}{0,.5,0}
\definecolor{webbrown}{rgb}{.6,0,0}

\usepackage{color}
\usepackage{fullpage}
\usepackage{float}

\usepackage{graphics}
\usepackage{latexsym}
\usepackage{epsf}
\usepackage{breakurl}

\def\gcd{\mathrm{gcd}}
\def\rank{\mathrm{rank}}
\def\modd#1 #2{#1\ \mbox{\rm (mod}\ #2\mbox{\rm )}}

\begin{document}

\begin{center}
\epsfxsize=4in

\end{center}

\theoremstyle{plain}
\newtheorem{theorem}{Theorem}
\newtheorem{corollary}[theorem]{Corollary}
\newtheorem{lemma}[theorem]{Lemma}
\newtheorem{proposition}[theorem]{Proposition}

\theoremstyle{definition}
\newtheorem{definition}[theorem]{Definition}
\newtheorem{example}[theorem]{Example}
\newtheorem{conjecture}[theorem]{Conjecture}

\theoremstyle{remark}
\newtheorem{remark}[theorem]{Remark}

\begin{center}
\vskip 1cm{\LARGE\bf  Lattice points close to a smooth curve and squarefull numbers in short intervals. II
}

\vskip 1cm
\large
Ognian Trifonov\\
University of South Carolina\\
Department of Mathematics \\
Columbia, SC 29208 \\
\href{mailto:trifonov@math.sc.edu}{\tt trifonov@math.sc.edu} \\
\end{center}

\vskip .2 in

\begin{abstract}
By applying recent estimates on the number of integer solutions of ternary quadratic forms in boxes  due to  Browning and Heath-Brown, we obtain new estimates for the number of lattice points close to a curve. This leads to a new result on the distribution of squarefull numbers in short intervals.
\end{abstract}

\section{Introduction}
A squarefull number is a positive integer $n$ such that if a prime $p$ divides $n$, then $p^2$ also divides $n$. Let $Q(x)$ be the number of squarefull numbers which do not exceed $x$. 

In 1958, Bateman and Grosswald \cite{BatGro} proved

\begin{equation}\label{asympt}
Q(x) \sim \frac{\zeta(3/2)}{\zeta(3)} x^{1/2} + \frac{\zeta(2/3)}{\zeta(2)}  x^{1/3} +  o\left (x^{1/6} \right ).
\end{equation}

Later on, P.~Shiu \cite{PShiu1980} showed that there exist infinitely many positive integers $n$ such that the interval $(n^2,(n+1)^2)$ does not contain squarefull numbers. Thus, he answered the question of how 
large  the gaps between squarefull numbers can be. 

Concerning the distribution of squarefull numbers, a natural question is, for what numbers $\theta \in [0,1/2)$ the following equation holds?
\begin{equation}\label{asympsh}
Q \left ( x + x^{1/2+\theta} \right ) - Q(x) \sim \frac{\zeta(3/2)}{2\zeta(3)} x^{\theta} 
\end{equation}.

The above result of Bateman and Grosswald shows that \eqref{asympsh} holds for all $\theta \in [1/6,1/2]$. P. Shiu's result shows that \eqref{asympsh} does not hold for $\theta = 0$. 
Over the years, smaller and smaller values of $\theta_0$ such that \eqref{asympsh} holds for all $\theta \in (\theta_0, 1/2)$ were obtained in a number of papers, P.~Shiu \cite{PShiu1984} (1984) $\theta_0=0.1526$, 
P.~G.~Schmidt \cite{PGSchmidt1986} (1986) $\theta_0=0.1507\ldots$, C.~H.~Jia \cite{CHJia} (1987) $\theta_0=0.149\ldots$, P.~G.~Schmidt \cite{PGSchmidt1988} (1988) $\theta_0=0.1428\ldots$, 
H.~Q.~Liu \cite{HQLiu} $\theta_0=0.1425\ldots$, D.~R.~Heath-Brown 
\cite{DRHeath-Brown1991} (1991) $\theta_0 = 0.1318\ldots$, M.~Filaseta and O.~Trifonov \cite{FilTri} (1994) $\theta_0=0.1282\ldots$, M.~N.~Huxley and O.~Trifonov \cite{HuxTri} (1996) $\theta_0=1/8=0.125$,
O.~Trifonov \cite{OTrifonov2002} (2002) $\theta_0 = 19/154 = 0.1233\ldots$.

We improve the above results further by showing:

\begin{theorem} \label{newtheta}
Let $\epsilon > 0$ and $\theta = \frac{8}{65} + \epsilon$. Then,
$$Q \left ( x + x^{1/2+\theta} \right ) - Q(x) \sim \frac{\zeta(3/2)}{2\zeta(3)} x^{\theta}.$$

\end{theorem}

Note that $\frac{8}{65} = 0.12307\ldots < \frac{19}{154} = 0.12337\ldots $.

The key ingredients in most of the above results have been estimates for the number of lattice points close to the curve $a^2b^3 = x$. 
The following theorem contains a new estimate for the number of lattice points close to a smooth curve which allows us to prove Theorem~\ref{newtheta}.

\begin{theorem} \label{newest}
Let $1 \leq C \leq M \leq T \leq M^2$, let $f: [M,2M] \to R$ have a continuous third derivative and ${\displaystyle \frac{T}{CM^j} \leq |f^{(j)} (x)| \leq \frac{CT}{M^j} }$ for $j=2,3$ and all $x \in [M,2M]$.  Let  $0 \leq \delta < 1/2$ and $\delta \leq (CM)^{-1/2}$. 
Define $S(f,\delta)$ to be the set of all integer points $(x,y)$ such that $x \in [M,2M]$ and $|f(x)-y| \leq \delta$.

Then, for every $\epsilon > 0$ there exists $c(\epsilon ) > 0$ such that  

$$
\begin{array}{lcl}
|S(f, \delta)| &  \ll  & c(\epsilon  ) M^\epsilon \left ( M^{12/25}T^{4/25} + M^{1/2}T^{1/7} + M^{7/8} \delta^{1/4}+M^{12/13} \delta^{4/13}+ M^{2/7}T^{2/7} \right )\\
 & & + M^{1/2}T\delta^{5/2}+M^{1/2}T^{1/4}\delta^{1/4} + TM^{-1}.
 \end{array}
$$
\end{theorem} 

Theorem \ref{newest} improves on previous results when $T= M^\alpha$ with $\alpha$ close to $3/2$ and $\delta $ sufficiently small.
For example, one gets the following corollary. 

\begin{corollary}
Let $f$, $C$, and $S(f,\delta)$ be as in Theorem \ref{newest}. Let $\delta < M^{-33/50}$ and $T = M^{3/2}$. Then,
$$|S(f,\delta)| \ll _\epsilon  M^{18/25+ \epsilon} \quad \mbox{ for every } \epsilon >0 .$$
\end{corollary}

The best previous result  was $|S| \ll M^{13/18}$ and $13/18 = 0.722\ldots > 18/25 = 0.72$. 

The proofs of the main results are similar to those in Trifonov \cite{OTrifonov2002} with two important differences. First, a new estimate of Heath-Brown 
for the number of zeros of a ternary quadratic form in a box plays a prominent role. Second, we obtain smaller secondary terms in Theorem \ref{newest} by using an estimate of Huxley and Sargos for the number of lattice points on quadratic major arcs. 

The paper is organized as follows. In Section 2 we present some properties of divided differences and their use to obtain bounds for the number of lattice points close to a smooth curve. We also explore connections between   Swinnerton-Dyer's paper  \cite{SwD1974} and divided differences.
In Section 3 we deal with convexity and 3-convexity of sets of lattice points close to a curve. In Section 4 we estimate the number of nondiagonal integer solutions of a certain system of two equations. Here the new result of Heath-Brown mentioned above plays a crucial role.
Section 5 is dedicated to the proof of Theorem \ref{newest} and Section 6 is devoted to to the proof of Theorem \ref{newtheta}.

\subsection{Notation}

$C$, $T$, and $M$ are real numbers such that $1 <  C  \leq M \leq T$ and $\delta$ is a real number in $(0, 1/8)$. 

Let $r$ be a positive integer. Define ${\mathcal F}_r$ to be the class of all real-valued functions which are defined on the interval $I = [M,2M]$, have a continuous $r$th derivative on $I$, and $T/(CM^r) \leq |f^{(r)} (x)| \leq CT/M^r$ for all $x \in I$. 

Define $S(f,\delta)$ to be the set of all integer points $(x,y)$ such that $x \in [M,2M]$ and ${\displaystyle |f(x)-y| < \delta}$. 

For any finite set of points in the plane $S$, we denote by ${\mathcal P}(S)$ its projection on the $x$-axis, and by $|S|$ its cardinality. 

We call the graph of any quadratic  function $f(x)=ax^2 + bx + c$ with $a$, $b$, and $c$ real numbers a parabola.

$f(u) \ll g(u)$ will mean that there exists an absolute constant $c$ such that $|f(u)| \leq cg(u)$ for all $u \geq 1$. $f(u) \ll_\epsilon  g(u)$ will mean that for every $\epsilon > 0$  there exists a constant $c(\epsilon)$ such that 
$|f(u)| \leq c(\epsilon )g(u)$ for all $u \geq 1$. Similarly, $f(u) \ll_C g(u)$ will mean that the implied constant in $\ll $ depends on $C$ only.

For $r$ and $n$ positive integers, we denote by $d_r(n)$ the number of distinct solutions in positive integers of the equation $x_1 x_2 \ldots x_r = n$.

For $x$ a real number we denote by $||x||$ the distance from $x$ to the closest integer.

To simplify the notation we adopt the convention $\epsilon \pm \epsilon = \epsilon$ and $c\epsilon = \epsilon$ whenever $c$ is an absolute constant. There will be no division by $\epsilon$ nor comparison of $\epsilon$'s in this paper.

\section{Divided Differences}

We use properties of divided differences in a number of places in this paper. Here we state some of the properties of divided differences we will need. The results in this section are well-known and can be found in most numerical analysis 
textbooks (for example, in \cite{IsKe}). 

Let $f: {\mathbb R} \to {\mathbb R}$ and let $x_0, x_1, \ldots, x_n$ be distinct real numbers. Then, there is a unique interpolating polynomial $P(x)$ of degree at most $n$ such that $P(x_k) = f(x_k)$ for $k=0,1,\ldots,n$. 
The coefficient of $x^n$ of $P(x)$ is the $n$th divided difference of $f$ at $x_0, x_1, \ldots, x_n$ and we denote it by $f[x_0, x_1, \ldots, x_n]$. Denote ${\displaystyle V = V(x_0, \ldots, x_n) = \prod_{0 \leq k < j \leq n} (x_j - x_k)}$ (that is, $V$ is the product of all distances between distinct points in the set $\{x_0, \ldots, x_n\}$). Also,  let 
$V_i = V(x_0, \ldots, x_{i-1},x_{i+1}, \ldots, x_n)$ for $i = 0, \ldots , n$ ($V_i$  is the product of all distances between distinct points in the set $\{x_0, \ldots, x_n\}/\{x_i\}$). 

\begin{lemma} \label{divdif1}
We have ${\displaystyle f[x_0, \ldots, x_n] = \frac{\sum_{i=0}^n (-1)^{n-i} V_i f(x_i)}{V}}$. 
\end{lemma}

\begin{lemma} \label{divdif2}
Let $m = \min (x_0, x_1, \ldots , x_n )$ and let $M = \max  (x_0, x_1, \ldots , x_n )$. Assume that $f$ has $n$th derivative on $[m,M]$. Then, there exists $\xi \in (m,M)$ such that 
$$f[x_0,x_1, \ldots , x_n] = f^{(n)} (\xi ) / n! .$$
\end{lemma}

\subsection{Swinnerton-Dyer's approach via divided differences}
In 1974, Swinnerton-Dyer wrote a paper \cite{SwD1974} on the number of lattice points on a convex curve.  
Although Swinnerton-Dyer did not use the term divided difference in \cite{SwD1974}, the main idea of the paper is related to divided differences. 

Consider a twice differentiable,  convex function
 $f :[M,2M] \to {\mathbb R}$ such that $\lambda_2 \geq f''(x) > 0$ for all $x$. 
Consider three lattice points on $y = f(x)$, 
 $A_0(x_0,f(x_0))$,  $A_1(x_1,f(x_1))$, and $A_2(x_2, f(x_2))$. 

Then,  $f[x_0, x_1, x_2] = \frac{2D}{(x_2-x_1)(x_2-x_0)(x_1-x_0)}$
 where $D$ is the area of  $\triangle A_0A_1A_2$. 

Since $D \geq 1/2$ and
$f[x_0,x_1,x_2]  = f''(\xi)/2 \leq \lambda_2 / 2$. 
we get $x_2 - x_0 \geq (2/\lambda_2)^\frac{1}{3}$, so each subinterval of $[M,2M]$ of  length $|I|$ contains $\ll |I|\lambda_2^{1/3} + 1$ lattice points. 

Thus,  if $\lambda_2 $ has size $1/M$ and $|I| = M$, we have  $\ll M^{2/3}$ lattice points on the curve $y=f(x)$. 

Note that if we have a lower bound better than $1/2$ for $D$, we will get a better estimate for $x_2-x_0$.

Swinnerton-Dyer's idea is  to consider four lattice points on $y=f(x)$. 
Let $A_0,A_1,A_2,A_3$ be lattice points with coordinates $(x_i,f(x_i))$ with $x_0 < x_1 < x_2 < x_3$
and let  $D_0, D_1, D_2, D_3$ be the areas of the four triangles determined by the points $A_0,A_1,A_2,A_3$. 

Let $A > 1/(\lambda_2^{1/3})$ and $B > 1/2$ be two parameters whose values will be determined later. 

There are $\ll M/A + 1$ quadruples of lattice points with $x_3 - x_0 > A$. 

Next, assume $x_3 -x_0 \leq A$. 
If  $D_i > B$ for $i=0,1,2$, or $3$, we consider the second divided difference with numerator $D_i$ and obtain as above, 
$x_3 - x_0 \gg (B/\lambda_2)^\frac{1}{3}$. Thus, there are $\ll M(\lambda_2/B)^\frac{1}{3}$ quadruples,  where for at least one $i$, $D_i > B$. 

Finally, consider the case when $x_3 - x_0 \leq A$ and $D_i \leq B$ for $i=0,1,2,3$. 
Swinnerton-Dyer showed that the above does not happen often when $f$ is in $C^3[M,2M]$,  $1/M \ll f''(x)\ll 1/M$ and 
$|f'''(x)| \ll 1/M^2$ for all $x \in [M,2M]$. 
He observed that
$D_1 - D_2 + D_3 - D_4 = 0$,
$x_1D_1 - x_2D_2 + x_3D_3 - x_4D_4 = 0$, and 
 $|x_1^2D_1 - x_2^2D_2 + x_3^2D_3 - x_4^2D_4| \ll  A^6/M^2$. 

When $A=M^{2/5}$ and $B =M^{1/5}$ he was able to show that there are $\ll_\epsilon M^{3/5+\epsilon}$
quadruples of lattice points satisfying the above system. With the above constraints on $f$, Swinnerton-Dyer showed that the curve $y=f(x)$, $x \in [M,2M]$ contains $\ll_\epsilon M^{3/5+\epsilon}$ lattice points.

The exponent $3/5$ has not been improved for the class of functions above. If one adds stronger smoothness conditions on $f$, much sharper bounds were obtained in the breakthrough work of Bombieri and  Pila \cite{BombPila} and Pila \cite{Pila}.

When $|f''|$ is not very small, for example, when $|f''|$ has size $M^{-1/2}$ one often gets better bounds by using third divided differences. It is natural to try to extend Swinnerton-Dyer's approach to third differences. 
There are two main difficulties in achieving this. First, we need to ensure that the numerators of the third divided differences are not zero. To achieve this we use work of Huxley and Sargos \cite{HuxSar} and a local 3-convexity lemma due to Konyagin. Second, we need to estimate the number of solutions of a certain system of equations in a box. To do this, we use the new estimate of Heath-Brown \cite{HeathBrown2024}.

\section{Three-convexity of a set of lattice points close to a curve}

\subsection{Preliminary lemmas}
Here we state two lemmas. The idea of the first lemma is that if some mild conditions on $\delta$ hold then one can find a subset of $S(f,\delta)$ which has the following properties:
\begin{enumerate}
\item it essentially has the same size as $S(f,\delta)$;
\item it is either strictly convex or strictly concave depending on the sign of $f''$;
\item no two elements of the subset are too close.
\end{enumerate}
The lemma below is contained in Huxley \cite[pp.~204--206]{Huxley1989}.

\begin{lemma} \label{conv}
Let $f\in\mathcal F_2$, $T>M$, and
${\displaystyle 
0<\delta^2<\frac14\min_{x\in I}|f''(x)|
}$. 

Then there exist non-intersecting sets $S_1$ and $S_2$ such that
\begin{enumerate}
\item $S(f,\delta)=S_1\cup S_2$;
\item $S_1$ is either strictly convex or strictly concave, depending on the sign of $f''$;
\item $|S_2|\ll C\delta M+1$.
\end{enumerate}
\end{lemma}

The next lemma is Lemma 5 in the paper \cite{OTrifonov2002}. 

\begin{lemma} \label{mindist}
Let $f\in\mathcal F_2\cap\mathcal F_3$ and $S_1\subset S(f,\delta)$, where $S_1$ is either strictly convex or strictly concave. Let $0<E<1/(16\delta)$. Define $S_3$ to be the set of integer points $(x_j,y_j)$ in $S_1$ such that there exist elements of $S_1$, say $(x_i,y_i)$ and $(x_k,y_k)$, with $x_i<x_j<x_k$, $x_j-x_i<E$, and $x_k-x_j<E$. Then
\[
|S_3|\ll \frac{C^3T\delta E^3}{M}+\frac{CTE^5}{M^2}.
\]
\end{lemma}

\subsection{The work of Huxley and Sargos on long major arcs}

Let $M > 1$, $0 < \delta \le 1/4$, and let $f: [M, 2M] \to {\mathbb R}$. 

Let $r \geq 2$ be an integer and assume $f$ is in $C^r$  and $c_r \lambda_r \leq |f^(r) (x)| < C_r \lambda_r$ for some $0 < c_r < C_r$, $0 < \lambda_r < 1$, and for all $x \in [M,2M]$. 

Let $S = \{ m \in [M,2M] \cap {\mathbb Z} | ||f(m)|| < \delta \}$. 

We say that ${\mathcal A} = \{ m_1, \ldots, m_N \}$ is a major arc if ${\mathcal A} \subseteq S$, $N \geq r^2+1$, and there exists a polynomial $P$ of degree less than $r$, such that $P(m_i) = ||f(m_i)||$ for $i=1,\ldots, N$. 

Lagrange's interpolation formula implies that $P(x)$ can be written as ${\displaystyle P(x) = \frac{1}{q} \sum_{j=0}^{r-1} b_j x^j}$ with $q \in {\mathbb N}$, as small as possible,  and $b_j \in  {'\mathbb Z}$ for $j=0, \ldots, r-1$. 
According to Huxley and Sargos, the integer $q$ above is the {\it denominator of the major arc}  ${\mathcal A}$.

Moreover, we have that
$q$ divides ${\displaystyle \prod_{1 \leq i < j \leq r} (m_j - m_i) }$. In particular,  $q < M^{\frac{r(r-1)}{2}}$. 

Huxley and Sargos \cite{HuxSar} define {\it long major arcs} as follows. 

Let $$H = r^2 + 2\max{1 \leq q \leq Q_0}(r-1)^{ \omega (q)},$$

where $\omega(q)$ is the number of distinct prime divisors of $q$ and $Q_0$ is the largest possible denominator of a major arc. We have $Q_0 <  M^{\frac{r(r-1)}{2}}$.

Huxley and Sargos \cite{HuxSar} show in eqn. (2.12) that

\begin{equation} \label{Hsize}
H \ll_\epsilon M^\epsilon \quad \mbox{ for every } \epsilon > 0.
\end{equation}

We say that ${\mathcal A}$ is a {\it long major arc } if the major arc contains at least $H$ elements of $S$, that is, $N  \geq H$. 

Let $R_H$ be the total number of elements of $S$ which are on one or more long major arcs. Combining \cite{HuxSar} Lemma 2 and equations (2.4) and (2.5), one gets the following lemma.

\begin{lemma}  \label{Rhest} \text{(Huxley and Sargos)}
Let $r \geq 2$. Then,
$$R_H \ll M \delta^{1/(r-1)} + \left ( \frac{\delta}{\lambda _r} \right ) ^{1/r} .$$
\end{lemma}

\subsection{Local 3-convexity}

We will need the following lemma which is Lemma 6 in \cite{OTrifonov2002}. The lemma is based on an idea of S.~Konyagin.

\begin{lemma} \label{local}
Let $k$ be a positive integer and let $T = \{(x_1,y_1), \ldots, (x_{4k}, y_{4k}) \}$ be a set of points in the plane with $x_1 < x_2 < \ldots  < x_{4k}$. Then, either there exists a parabola which passes through at least $k$ elements of $T$, or there exists a subset $U$ of five distinct points of $T$ such 
that no parabola passes through any four points in $U$. 
\end{lemma}

\section{The number of non-diagonal integer solutions of a certain system of two equations}

\subsection{The Heath-Brown estimate on the number of zeros of ternary quadratic forms in boxes}

The following result of Heath-Brown \cite{HeathBrown2024} plays a crucial role in this paper. The result and a sketch of the proof were communicated to the author by D.~R.~Heath-Brown in 2024.

For any non-zero integer $M$, let $M^\#  = \prod_{p^e || M, e \geq 2} p^e$ denote the (positive) squarefull part of $M$.

\begin{lemma} \label{tern} (Heath-Brown) 
Let $q$ be a non-singular integral ternary quadratic form
with matrix ${\mathcal M}$. Let $\Delta = |\det ({\mathcal M})|$ and assume $\Delta \neq 0$. Write $\Delta_0$ for the highest common factor of the $2 \times 2$ minors of ${\mathcal M}$.  
Then the number of primitive integer solutions of $q({\bf x}) = 0$ in the box $|x_i| \leq R_i$ is 
\begin{equation} \label{ubtqf}
\ll _\epsilon \left \{ 1 + \left ( \frac{R_1 R_2 R_3 \left ( \Delta_0^\# \right ) ^{3/2}}{\Delta} \right )^{1/3}  \right \} d_3(\Delta ),
\end{equation}
for any $\epsilon > 0$.
\end{lemma}

\begin{proof}
Let $q$ be a nonsingular integral ternary quadratic form with discriminant ${\bf \Delta}_q$. We use the following lemma which is Lemma 2.6 in the paper of Browning and Heath-Brown  \cite{BH2018}.

First, we need to define a notation used by Browning and Heath-Brown. 

For any prime $p$ let $\overline{q}$ denote the reduction of $q$ modulo $p$. Define a completely multiplicative function $\chi _q : {\mathbb N} \to \{ 0, \pm 1 \}$, via
$$\chi_q (p) = \left \{  \begin{array}{ll}
+1,& \quad \text{ if } \rank \; \overline{q} =2 \text{ and } \overline{q} \text{ is reducible over } {\mathbb F}_p, \\
-1,&\quad  \text{ if } \rank \;  \overline{q} =2 \text{ and } \overline{q} \text{ is irreducible over } {\mathbb F}_p,  \\
0,& \quad \text{ if } \rank \;  \overline{q}  \neq 2.
\end{array} \right .$$

\begin{lemma} (Browning and Heath-Brown) \label{basis}
Let $q$ be a non-singular  ternary quadratic form over ${\mathbb Z}$ with matrix ${\bf A}$.  Let $ {\bf \Delta}_q = \det {\bf A}$ and let ${\bf \Delta}_0$ be the highest common factor of the $2 \times 2$ minors of ${\bf A}$. Then there are
lattices $\Lambda _i$, $1 \leq i \leq I$ such that
$$\{{\bf  y} \in {\mathbb Z}_{prim}^3 \ : \ q({\bf y}) = 0 \} \subseteq \bigcup\limits_{i=1}^I \Lambda _i.$$
Moreover we have
$$\det \Lambda_i \gg \frac{|{\bf \Delta}_q|}{\left ( {\bf \Delta}_0^\# \right )^{3/2}} $$
for all $i$, and $I \ll C(q)$, where
$$C_q = \prod_{p^\xi || {\bf \Delta}_q, p | 2D(q)} d(p^\xi) \prod_{p^\xi || {\bf \Delta}_q, p \nmid 2D(q)} \left \{ \sum_{k=0}^\xi \chi (p^k)   \right \}.$$
\end{lemma}

Let $\Lambda$  be one of the lattices $\Lambda _i$ whose union contains the primitive solutions of ${\bf q}(x)=0$. Define a new lattice, $\Lambda ^*$ by 
$$\Lambda^* = \{ (x_1/R_1, x_2/R_2, x_3/R_3) : (x_1,x_2,x_3) \in \Lambda \}.$$

Then, 
\begin{equation} \label{dlt*}
\det (\Lambda ^*) = \det (\Lambda)/(R_1 R_2 R_3) \gg \frac{|{\bf \Delta}_q|}{\left ( {\bf \Delta}_0^\# \right )^{3/2}R_1R_2R_3}.
\end{equation}
Note that if $(x_1, x_2, x_3) \in \Lambda$ is such that $|x_i| \leq R_i$, $i=1,2,3$, then $(x_1/R_1, x_2/R_2, x_3/R_3) \in \Lambda ^*$ satisfies $|x_i/R_i| \leq 1$ for $i=1,2,3$. 

Next, we need a lemma which is Lemma 2.3 in \cite{BH2018}.

\begin{lemma} (Browning and Heath-Brown)  \label{basis1}
Let $\Lambda \subseteq {\mathbb R}^n$ be a lattice of dimension $k \leq n$. Then there exists a basis ${\bf g}^{(1)}, \dots , {\bf g}^{(k)}$ of $\Lambda $ for which 
$$\prod_{j=1}^k  |{\bf g}^{(j)}| \geq \det \Lambda , $$
and such that if ${\bf x} \in {\mathbb R}^n$ can be written as 
$${\bf x} = \sum_{j=1}^k c_j {\bf g}^{(j)}, $$
then $$|c_j| \leq n^{2n} |{\bf x}|/ | {\bf g}^{(j)}|.$$ 
\end{lemma}

Let $g^{(1)},g^{(2)},g^{(3)}$ (with $g^{(i)}\in {\mathbb R}^3)$ be one of the bases for $\Lambda ^*$ with the properties described in Lemma \ref{basis1}. So, we are trying to get a bound for the number of solutions of 
$q({\bf x} ) =0$ of the form ${\bf x}=c_1 g^{(1)} + c_2 g^{(2)} + c_3 g^{(3)}$ which are in the box  $|c_j| \leq 3^6 |{\bf x}|/ | {\bf g}^{(j)}|$, $j=1,2,3$ (with $|{\bf x}| \leq \sqrt{3}$ since the components of ${\bf x}$ have absolute values $\leq 1$) . Note that $q(c_1 g^{(1)} + c_2 g^{(2)} + c_3 g^{(3)}) $ is 
an integral non-singular quadratic form in $c_1,c_2, c_3$. Denote it by $q_1(c_1,c_2,c_3)$. Moreover, if $(c_1,c_2,c_3)$ is non-primitive triple, so is $(x_1,x_2,x_3)$ where $(x_1,x_2,x_3)={\bf x}$. 
Next, we will use Lemma 2.5 of the paper of Browning and Heath-Brown \cite{BH2018}.

\begin{lemma} \label{solbox}
 Let $q(x_1,x_2,x_3)$ be a non-singular integral quadratic form. Let $L_1, L_2, L_3 > 0$. Then there are 
$O\left ( 1  + (L_1L_2L_3)^{1/3} \right )$ primitive integer solutions to $q(x_1,x_2,x_3)=0$.
\end{lemma}

We apply Lemma \ref{solbox} to $q_1$ with $L_j = 729 \sqrt{3} /|g^{(j)}|$, and obtain the bound $$O\left ( 1 + 1/(\prod_{j=1}^3  |{\bf g}^{(j)}|)^{1/3} \right ).$$ 
By Lemma \ref{basis} $\prod_{j=1}^3  |{\bf g}^{(j)}| \geq \det(\Lambda ^*)$. Using equation \eqref{dlt*} and the fact that $C(q) \leq d_3({\bf \Delta })$, we complete the proof of the lemma.

\end{proof}

\subsection{The system of two equations} \label{twoeq}
Let $A > 1$ and $B > 1$ be real numbers. We get an upper bound for the number of integer solutions of the system
\begin{equation} \label{sys}
\begin{array}{rcl}
D_1 + D_3& \neq &D_2 + D_4\\
z_1D_1 + z_3D_3& = &z_2D_2 + z_4D_4\\
z_1^2D_1 + z_3^2D_3& = &z_2^2D_2 + z_4^2D_4,
\end{array}
\end{equation}
which satisfy the conditions
\begin{equation} \label{cond1}
0 < z_1 < z_2 < z_3 < z_4 \leq A, \quad 0 < |D_j| \leq B \ \mbox{ for } j=1,2,3,4,
\end{equation}
and 
\begin{equation} \label{cond2}
\gcd(z_1,z_2,z_3) | \gcd(D_1,D_2,D_3,D_4).
\end{equation}

Note that without the conditions  $D_1 +D_3 \neq D_2 + D_4$ and \eqref{cond1}  the system \eqref{sys} 
has $\gg A^2 B^2$ solutions, for example $D_1=D_2$, $z_1=z_2$, $D_3=D_4$,  $z_3=z_4$. We are trying to estimate the number of non-diagonal solutions of \eqref{sys}. 

Based on a heuristic we conjecture that if $A \geq B$, the number of solutions of the system  \eqref{sys} is

$$ \ll _\epsilon A^{1+\epsilon}B^{2+\epsilon}.$$

\begin{lemma} \label{nond}
The number of integer solutions of the system \eqref{sys} which satisfy \eqref{cond1},  \eqref{cond2}, and $\gcd(z_1,z_2,z_3)=d$ is
$$ \ll _\epsilon \frac{A^{1+\epsilon}B^{2+\epsilon}}{d^2} + \frac{B^{4 + \epsilon}}{d^4}.$$
\end{lemma}

To prove Lemma \ref{nond} we need the following auxiliary lemma.
For a positive integer $n$, denote by $s(n)$ the squarefree part of $n$, that is $s(n) = n/ n^{\#}$.

\begin{lemma} \label{ser}
Let $x \geq 2$. Then, 
$$\sum_{n \leq x} \frac{(n^{\#})^{1/2}}{n^{4/9}} \ll x^{5/9} \log x.$$
\end{lemma}

\begin{proof}
Denote the set of squarefull integers by ${\mathcal R}$ and the set of squarefree integers by ${\mathcal U}$. We have
$$S(x):= \sum_{n \leq x}\frac{(n^{\#})^{1/2}}{n^{4/9}} =  \sum_{r \in {\mathcal R}, r \leq x} \sum_{n^{\#} = r, n \leq x} \frac{(n^{\#})^{1/2}}{n^{4/9}}.$$
Therefore, since $n = s(n) n^{\#}$, 
$$S(x) =  \sum_{r \in {\mathcal R}, r \leq x} r^{1/2} \sum_{s \in {\mathcal U}, \gcd(s,r)=1, s \leq x/r} \frac{1}{(sr)^{4/9}} \leq \sum_{r \in {\mathcal R}, r \leq x} r^{1/18} \sum_{s \leq x/r} \frac{1}{s^{4/9}} \ll 
\sum_{r \in {\mathcal R}, r \leq x} r^{1/18} \left ( \frac{x}{r} \right )^{5/9}. $$
Thus, 
$$S(x) \ll x^{5/9} \sum_{r \in {\mathcal R}, r \leq x} \frac{1}{r^{1/2}}.$$

Now, we claim that  ${\displaystyle \sum_{r \in {\mathcal R}, r \leq x}  \frac{1}{r^{1/2}} \ll \log x}$. Indeed, if $A \geq 1$, 
$$\sum_{r \in {\mathcal R}, A \leq r <  2A}  \frac{1}{r^{1/2}} \leq \frac{1}{A^{1/2}} \sum_{r \in {\mathcal R}, A \leq r <  2A} 1 \ll 1,$$
where we have used Bateman's result that there are $\ll A^{1/2}$ squarefull numbers which are $\leq 2A$. Letting $A$ run through powers of $2$, $2^k \leq x$, adding the resulting inequalities, we get
${\displaystyle \sum_{r \in {\mathcal R}, r \leq x}  \frac{1}{r^{1/2}} \ll \log x}$. 
\end{proof}

Next, we prove Lemma \ref{nond}.

\begin{proof}
We start as in \cite{OTrifonov2002}. From \eqref{sys} we have $$z_4 = (z_1D_1 - z_2D_2 + z_3D_3)/D_4.$$

Substituting the expression for $z_4$ in the last equation of \eqref{sys} and simplifying we get

\begin{eqnarray}
(D_1^2 - D_1D_4)z_1^2 + (D_2^2 + D_2D_4)z_2^2 + (D_3^2 - D_3D_4)z_3^2  \label{tf} \\
-2D_1D_2z_1z_2 + 2D_1D_3z_1z_3 - 2D_2D_3 z_2z_3 = 0. \nonumber
\end{eqnarray}

The left-hand-side  of the above equation is a ternary quadratic form, say $q$,  in $z_1, z_2, z_3$.  The absolute value of the determinant of $q$ is $${\bf \Delta }= |D_0D_1D_2D_3|D_4^2,$$
where we have denoted $D_0 = (D_1 + D_3) - (D_2 + D_4)$. 

As in \cite{OTrifonov2002} we have
\begin{equation} \label{dlt0}
{\bf \Delta}_0 = |D_4| \gcd( D_0D_1D_2, D_0D_1D_3, D_0D_2D_3, D_1D_2D_3).
\end{equation}

We consider $D_1,D_2,D_3,D_4$ as fixed (note that $D_1,D_2,D_3,D_4$ uniquely determine the value of $D_0$) and estimate the number of integer solutions $S(D_1,D_2,D_3,D_4)$, 
$(z_1,z_2,z_3)$ of \eqref{tf} which satisfy $0 < z_i \leq A$, for $i=1,2,3$.

For $a \in {\mathbb N}$ with $a \leq B$, let $N_a$ be the set of quadruples $(D_1,D_2,D_3,D_4)$ where $D_i$ satisfy the conditions of Lemma \ref{nond} and $\gcd(D_1,D_2, D_3,D_4)=a$,  that is
$$
N_a = \{ (D_1,D_2,D_3,D_4)\;  | \; 0< |D_i| \leq B, \text{ for } i=1,2,3,4, D_1 + D_3 \neq D_2+D_4,$$ $$ \text{ and } \gcd(D_1,D_2, D_3,D_4)=a \}.
$$

First, we will consider the case $(D_1,D_2,D_3,D_4) \in N_1$. Then $\gcd(D_0,D_1,D_2,D_3) = 1$ (since $D_0 = (D_1 + D_3) - (D_2 + D_4)$). Moreover, \eqref{cond2} implies $\gcd(z_1,z_2,z_3)=1$, so we will
be counting primitive solutions of $q({\bf z})=0$.

By Lemma \ref{tern} we have
\begin{equation} \label{ND}
S(D_1,D_2,D_3,D_4) \ll d_3(\Delta)  \left \{ 1 + A\frac{({\bf \Delta}_0^{\#})^{1/2}}{{\bf \Delta}^{1/3}}\right \}.
\end{equation}

Next we split $N_a$ into four (possibly intersecting) parts. For $i \in \{0,1,2,3 \}$, define 
$$N_{a,i} = \{ (D_1,D_2,D_3,D_4) \in N_a \ |\  D_i^4 \geq |D_0D_1D_2D_3| \}.$$

Note that we get a contradiction if $D_i^4 < |D_0D_1D_2D_3|$ for $i=0,1,2,3$. 

Thus, ${\displaystyle N_a \subseteq \bigcup_{i=0}^3 N_{a,i}}$.










For $1 \leq a \leq B$, define $$S_a := \sum_{(D_1, D_2, D_3,D_4) \in N_a} S(D_1, D_2, D_3,D_4) \text{ and } S_{a,i} := \sum_{(D_1, D_2, D_3,D_4) \in N_{a,i}} S(D_1, D_2, D_3,D_4) .$$

Clearly, $$S_a \leq \sum_{i=0}^3 S_{a,i}.$$

Next, we estimate $S_{1,3}$. The estimation of $S_{1,i}$ for $i=0,1,2$ is similar. 

Recall that ${\bf \Delta}_0 = |D_4| \gcd( D_0D_1D_2, D_0D_1D_3, D_0D_2D_3, D_1D_2D_3)$. 

Therefore, ${\bf \Delta}_0 \mid D_0D_1D_2D_4$, so $({\bf \Delta}_0)^{\#} \mid |D_0D_1D_2D_4|^{\#}$.

For $G \geq 1$ and $i \in \{0,1,2,3 \}$ define $S_{a,i,G} = \{ (D_1, D_2, D_3, D_4) \in N_{a,i}\;  | \; G \leq D_4 < 2G \}$.

Let $ (D_1, D_2, D_3, D_4) \in N_{1,3,G}$. Then $G \leq |D_4|  < 2G$ and $|D_3| \geq |D_0D_1D_2|^{1/3}$. 

Therefore, 
$$
\frac{({\bf \Delta}_0^{\#})^{1/2}}{{\bf \Delta}^{1/3}} \leq \frac{( |D_0D_1D_2D_4|^{\#})^{1/2}}{(|D_0D_1D_2D_3|D_4^2)^{1/3}} \leq \frac{( |D_0D_1D_2D_4|^{\#})^{1/2}}{(|D_0D_1D_2|^{4/3}|D_4|^{4/3})^{1/3}G^{2/9}}.
$$

Hence,
\begin{equation} \label{ratb}
\frac{({\bf \Delta}_0^{\#})^{1/2}}{{\bf \Delta}^{1/3}} \leq G^{-2/9} \frac{( |D_0D_1D_2D_4|^{\#})^{1/2}}{ |D_0D_1D_2D_4|^{4/9}}.
\end{equation}

Since ${\bf \Delta} \leq  4B^6$, $d_3({\bf \Delta}) \ll _\epsilon B^\epsilon$ and using equation \eqref{ND} we get
\begin{equation} \label{S13}
S_{1,3} \ll _\epsilon B^\epsilon \sum_{(D_1, D_2, D_3,D_4) \in N_{1,3}}  \left  \{ 1 + A\frac{({\bf \Delta}_0^{\#})^{1/2}}{{\bf \Delta}^{1/3}}\right \} = B^\epsilon |N_{1,3}| + AB^\epsilon \sum_{(D_1, D_2, D_3,D_4) \in N_{1,3}}\frac{({\bf \Delta}_0^{\#})^{1/2}}{{\bf \Delta}^{1/3}}
\end{equation}

We have
\begin{equation} \label{S131}
 \sum_{(D_1, D_2, D_3,D_4) \in N_{1,3}}\frac{({\bf \Delta}_0^{\#})^{1/2}}{{\bf \Delta}^{1/3}} \leq  \sum_{G=2^k, 2^k \leq B}  \sum_{(D_1, D_2, D_3,D_4) \in N_{1,3,G}}\frac{({\bf \Delta}_0^{\#})^{1/2}}{{\bf \Delta}^{1/3}}
\end{equation}

Now, using the equation \eqref{ratb} we get 
\begin{equation} \label{S132}
 \sum_{(D_1, D_2, D_3,D_4) \in N_{1,3}}  \frac{({\bf \Delta}_0^{\#})^{1/2}}{{\bf \Delta}^{1/3}} \leq  \sum_{G=2^k, 2^k \leq B} G^{-2/9}  \sum_{(D_1, D_2, D_3,D_4) \in N_{1,3,G}}  \frac{( |D_0D_1D_2D_4|^{\#})^{1/2}}{ |D_0D_1D_2D_4|^{4/9}}.
\end{equation}

Note that when $(D_1, D_2, D_3,D_4) \in N_{1,3,G}$ we have $ |D_0D_1D_2D_4| \leq 4B^3G$. Also for a positive integer $t \leq 4B^3G$ the equation $|D_0D_1D_2D_4|  = t$ has $\ll_\epsilon B^\epsilon$ solutions. 

Therefore,
\begin{equation} \label{S133}
 \sum_{(D_1, D_2, D_3,D_4) \in N_{1,3}}  \frac{({\bf \Delta}_0^{\#})^{1/2}}{{\bf \Delta}^{1/3}} \ll _\epsilon B^\epsilon   \sum_{G=2^k, 2^k \leq B} G^{-2/9} \sum_{t=1}^{4B^3G} \frac{(t^{\#})^{1/2}}{t^{4/9}}.
\end{equation}

Applying Lemma \ref{ser} with $x=4B^3G$, we get 
 \begin{equation} \label{S134}
 \sum_{(D_1, D_2, D_3,D_4) \in N_{1,3}}  \frac{({\bf \Delta}_0^{\#})^{1/2}}{{\bf \Delta}^{1/3}} \ll _\epsilon B^\epsilon   \sum_{G=2^k, 2^k \leq B} G^{-2/9} (B^3G)^{5/9} \log B \ll _\epsilon B^{2+2\epsilon},
\end{equation}
where we have used $\log B \ll _\epsilon B^\epsilon$.

Combining \eqref{S13} and \eqref{S134} we obtain
\begin{equation} \label{S135}
S_{1,3} \ll _\epsilon B^\epsilon |N_{1,3}|  + AB^{2 + 3\epsilon}.
\end{equation}

One obtains similarly 
$$S_{1,i} \ll _\epsilon B^\epsilon |N_{1,i}|  + AB^{2 + 3\epsilon},$$
  when $i=0,1,2$. 

Therefore, 
\begin{equation} \label{S136}
S_1 \ll _\epsilon B^\epsilon |N_1|  + AB^{2 + 3\epsilon}.
\end{equation}

Next, we need estimates for $S_a$ when $a>1$.  We proceed in the same way as above with two differences. First, since $\gcd(D_1, D_2, D_3,D_4)=a$, we can replace each $D_i$ in \eqref{tf} by $D_i' = D_i/a$ for $i=0,1,2,3,4$, and the equation still holds.
Now, we have $|D_i'| \leq B/a$ for $i=0,1,2,3,4$. 
We get that the number of {\bf primitive} solutions of $q(z_1,z_2,z_3)=0$ when $(D_1, D_2, D_3,D_4) \in N_a$ is $\ll _\epsilon B^\epsilon |N_1|  + A\left ( \frac{B}{a} \right )^{2 + 3\epsilon}$.

However, $\gcd(D_1, D_2, D_3,D_4)=a$ and condition \eqref{cond2} only gives us $\gcd(z_1,z_2,z_3) | a$. If $(z_1,z_2,z_3)$ is not a primitive solution with $\gcd(z_1,z_2,z_3) = b$, then $(z_1/b,z_2/b,z_3/b)$ is a primitive solution and 
$b | a$. So, each primitive solution is associated to at most $d(a) \ll _\epsilon B^\epsilon $ non-primitive solutions. 

Thus, 
\begin{equation} \label{S137}
S_a \ll _\epsilon B^{2\epsilon} |N_a|  + A\left ( \frac{B}{a} \right )^{2 + 3\epsilon}B^\epsilon
\end{equation}

Using that $|N_a| \leq (B/a)^4$ we complete the proof of the lemma.


\end{proof}

\section{The proof of Theorem \ref{newest}}

\begin{proof}
The proof of this theorem will follow closely the proof in \cite{OTrifonov2002}. In places which are very technical and exactly coincide with the proof in \cite{OTrifonov2002} we refer to the relevant pages of that paper. 

Our strategy will be to split $S(f,\delta)$ into a nice part and a remainder. The nice part will  locally have a property related to 3-convexity  and will be such that no two of its elements are very close. Also $|S(f, \delta)| \ll_\epsilon M^\epsilon |\text{  nice part }| + |\text{ remainder }|$, 
and the size of the remainder will be relatively small. To estimate the size of the nice part we apply the approach of Swinnerton-Dyer extended to third divided differences. 

First, by Lemma \ref{conv},
$
S(f,\delta)=S_1\cup S_2,
$
where $S_1$ is either strictly convex or strictly concave, depending on the sign of $f''$, and $|S_2|\ll \delta M+1$.
Therefore, 
\begin{equation} \label{S2}
|S(f, \delta)| \ll |S_1| + \delta M+1.
\end{equation}

Let $E = \min (1/(17\delta), C(M\delta)^{1/2})$.  Since $1 \leq C \leq M$ and $\delta < (CM)^{-1/2}$ we have 
$E \ll C(M\delta)^{1/2}$.

We apply Lemma \ref{mindist}  to $S_1$ and obtain a subset $S_3$ containing the elements of $S_1$ which are close to both of their neighbors. (The point $(x_j,y_j)$ is in $S_3$ if $|x_j - x_{j-1}| < E$ and $|x_j - x_{j+1}| < E$.)  We discard all elements of $S_3$ from $S_1$, and then discard every other element of the remaining set. We denote what is left by $S_4$. Clearly $|S_1|\le 2|S_4|+|S_3|$, and if $(x_i,y_i)$ and $(x_j,y_j)$ are distinct elements of $S_4$, then $|x_j-x_i|>E$.

By Lemma \ref{mindist} we have
\begin{equation} \label{S3est}
|S_3| \ll M^{1/2}T\delta^{5/2}.
\end{equation}
We obtain, 
\begin{equation} \label{S4}
|S(f, \delta)| \ll |S_4| + \delta M+  M^{1/2}T\delta^{5/2}+1.
\end{equation}

Let $H$ be the number defined in Section 3.2 with $r=3$.
We split the set $S_4$ into groups of $4H$ consecutive elements (where the last group may be shorter).

We say that a group is of type 1 if at least $H$ elements of the group are on the same parabola, a long major arc (with $r=3$). By the lemma of Huxley and Sargos, Lemma \ref{Rhest}, 
we have that the contribution to $S_4$ from such groups is at most $4R_H \ll M \delta^{1/2} + \left ( {\delta}/{\lambda _3} \right ) ^{1/3} $. Let $S_5$ be the union of all groups of type 1 in $S_4$. 

Since $\lambda_3 \gg T/M^3$ we have that $|S_5| \ll M \delta^{1/2} +M (\delta / T)^{1/3}$. 

If a group is not of type 1 (we will say it is of type 2), then by Lemma \ref{local} it contains a quintuple $\{ (x_j,y_j) | j=1,\ldots,5 \}$ such that no four distinct points of the quintuple are on the same parabola.
W.l.o.g. we can assume $x_1 < x_2 < x_3 < x_4 < x_5$.  Let $S_6$ be the union of all groups of type 2, and let $S_7$ be the subset of $S_6$ containing the above quintuples (one for each type 2 group). 

Thus, $|S_4| \leq |S_5| + |S_6|$ and $|S_6| \leq (4H/5)|S_7| \ll_\epsilon M^\epsilon |S_7|$. (Recall that $H \ll_\epsilon M^\epsilon$.) 
Combining the above inequalities with equation \eqref{S4} we obtain
\begin{equation} \label{S7}
|S(f, \delta)| \ll |S_6| +  M^{1/2}T\delta^{5/2}+M \delta^{1/2} +M (\delta / T)^{1/3}+1.
\end{equation}

We omitted a $\delta M$ term  since $M \delta^{1/2} > M\delta$. 

The set $S_7$ has the necessary properties we need for the extension of the Swinnerton-Dyer's approach. The set is either strictly convex or strictly concave, the $x$-coordinates of any two points in the set differ at least by $E$, and each quintuple in the set is such that no four distinct points of the quintuple are on the same parabola.

Next, we work on estimating $|S_7|$.

Consider a quintuple of consecutive integer points in $S_7$,
$
\{(x_i,y_i):1\le i\le 5\}.
$

Let $A$ be a positive parameter. There are at most $M/A+1$ quintuples with $x_5-x_1>A$.

Let $S'_7$ consist of quintuples in $S_7$ with  $x_5-x_1\le A$. Then,
\begin{equation} \label{S'7}
|S_7| \leq |S'_7| + \frac{5M}{A} + 5.
\end{equation}

Next, we  estimate $|S'_7|$. Define $a_i=x_{i+2}-x_{i+1}$ and $b_i=y_{i+2}-y_{i+1}$ for $i=0,1,2,3$. Then $a_0+a_1+a_2+a_3\le A$ and each $a_i \geq E$.

Consider the matrix
\[
{\mathcal R}=\begin{pmatrix}
1&1&1&1&1\\
x_1&x_2&x_3&x_4&x_5\\
x_1^2&x_2^2&x_3^2&x_4^2&x_5^2\\
x_1^3&x_2^3&x_3^3&x_4^3&x_5^3\\
y_1&y_2&y_3&y_4&y_5
\end{pmatrix}.
\]
Define $D_{s-1}=\det({\mathcal R}_{4,s})$ for $s=1,\ldots ,5$, where ${\mathcal R}_{4,s}$ is the $(4,s)$ minor of ${\mathcal R}$. Note that each $D_i$ is the numerator of a third divided difference. If we define $g(x_i)=y_i$ for $i=1,2,3,4,5$, then, for example $D_4$ is the numerator of the divided difference $g[x_1,x_2,x_3,x_4]$. Each $D_s$ is nonzero because no four points of the quintuple lie on a parabola.

This gives a mapping ${\mathcal H}$  from the quintuples in $S'_7$ to thirteen-tuples $$(a_0,\dots,a_3,b_0,\dots,b_3,D_0,\dots,D_4).$$ Note that the mapping ${\mathcal H}$  is one-to-one. Indeed, if the $i$th and $j$th quintuples map to the same thirteen-tuple
$(a_0,\dots,a_3,b_0,\dots,b_3,D_0,\dots,D_4)$, then 
$$\frac{y_2^{(i)} - y_1^{(i)}}{x_2^{(i)} - x_1^{(i)}}= \frac{b_0}{a_0} = \frac{y_2^{(j)} - y_1^{(j)}}{x_2^{(j)} - x_1^{(j)}},$$
which contradicts the strong convexity (concavity) of $S_7$. 

The identities
\[
\begin{array}{rcl}
D_0-D_1+D_2-D_3+D_4&\ =\ &0,\\
x_1D_0-x_2D_1+x_3D_2-x_4D_3+x_5D_4&\ =\ &0,\\
x_1^2D_0-x_2^2D_1+x_3^2D_2-x_4^2D_3+x_5^2D_4&\ =\ &0,\\
x_1^3D_0-x_2^3D_1+x_3^3D_2-x_4^3D_3+x_5^3D_4&\ =\ &-\det({\mathcal R})
\end{array}
\]
hold. After translating by $x_1$ and setting $z_j=x_{j+1}-x_1$ for $j=1,2,3,4$, we obtain a system of the form
\[
\begin{array}{rcl}
-z_1D_1+z_2D_2-z_3D_3+z_4D_4&\ =\ &0,\\
-z_1^2D_1+z_2^2D_2-z_3^2D_3+z_4^2D_4&\ =\ &0, \text{ and }\\
-z_1^3D_1+z_2^3D_2-z_3^3D_3+z_4^3D_4&\ =\ &-\det({\mathcal R}).
\end{array}
\] 
Further, one can get an upper bound for $D_j$, $j=0,\ldots,4$. For example, from 
Lemma \ref{divdif1} it follows that 
$$f[x_1,x_2,x_3,x_4]V(x_1,x_2,x_3,x_4) = D_4 + \theta \delta A^3,$$
where $|\theta| < 4$. 
Therefore, 
\[
0<|D_4|\le B:=A^6\frac{CT}{M^3}+4\delta A^3.
\]
Working similarly we obtain 
\begin{equation} \label{Dbound}
0<|D_j|\le B:=A^6\frac{T}{M^3}+4\delta A^3 \text{ for } j=0,\ldots ,4.
\end{equation}

Let $d= \gcd(z_1,z_2,z_3)$. From the definition of $z_1,z_2,z_3$ we have $d  \mid  \gcd(a_0,a_1,a_2)$.
Next, we show that $d\mid D_j$ for $j=0,\dots ,4$.

Using the definition of $D_3$ we obtain 
$$D_3=a_0(a_0+a_1)(a_1(b_2+b_3)-(a_2+a_3)b_1)-(a_2+a_3)(a_1+a_2+a_3))(a_0b_1-a_1b_0).$$
Thus, $d|D_3$. We show similarly that $d|D_i$ for $i=1,2,4$ (see \cite[p.314]{OTrifonov2002}  for details.) 

Since $0 \neq D_0 = D_1 - D_2 + D_3 - D_4$ we have $d \mid D_0$ and $D_1 + D_3 \neq D_2 + D_4$. 

Consider the map defined on ${\mathcal H}(S'_7)$ by
$$p:(a_0,\dots,a_3,b_0,\dots,b_3,D_0,\dots,D_4) \to (z_1,z_2,z_3,z_4,D_1,D_ 2,D_3,D_4).$$

Note that each octuple $(z_1,z_2,z_3,z_4,D_1,D_ 2,D_3,D_4)$ in the image of $p$ is a solution of the system of equations \eqref{sys} with conditions \eqref{cond1} and \eqref{cond2},  see Section \ref{twoeq}.

By Lemma \ref{nond} if $\gcd(z_1,z_2,z_3)=d$  the system of equations \eqref{sys} with conditions \eqref{cond1} and \eqref{cond2} has 
$$\ll _\epsilon \frac{A^{1+\epsilon}B^{2+\epsilon}}{d^2} + \frac{B^{4 + \epsilon}}{d^4}$$
solutions. This gives us an upper bound for the image of $p$, $p({\mathcal H}(S'_7))$.

The next lemma (contained in \cite{OTrifonov2002}) gives an upper bound for the size of the preimage of an octuple in 
$p({\mathcal H}(S'_7))$.

\begin{lemma} \label{preim}
Let
\begin{equation} \label{Abound}
A \leq \min \left(  \frac{1}{8\delta}, \frac{1}{2}M^{5/7}(CT)^{-2/7}, \frac{1}{3}M^{1/2}(\delta CT)^{-1/4},  M^2/(CT) \right).
\end{equation}
Furthermore, let $(z_1,z_2,z_3,z_4,D_1,D_ 2,D_3,D_4)$ be an octuple in $p({\mathcal H}(S'_7))$ with $\gcd(z_1,z_2,z_3)=d$. 
Then, there are $\ll d$ thirteen-tuples $(a_0,\dots,a_3,b_0,\dots,b_3,D_0,\dots,D_4)$ which map to 
$(z_1,z_2,z_3,z_4,D_1,D_ 2,D_3,D_4)$.
\end{lemma}

See \cite{OTrifonov2002} pp.314---316 for the proof of the lemma. The proof of Lemma \ref{preim} is the only place where the condition that 
the $x$-coordinates of any two distinct points in $S_7$ differ by at least $E$ is used in this paper. 

Summing over $d \leq A$ we obtain that $$| {\mathcal H}(S'_7)| \ll _\epsilon A^{1+\epsilon}B^{2+\epsilon}+ B^{4 + \epsilon}.$$

Since ${\mathcal H}$ is one-to-one we get the same upper bound on $|S'_7|$. 

Combining the above with equation \eqref{S'7} we obtain 
\begin{equation} \label{S7b}
|S_7| \ll _\epsilon A^{1+\epsilon}B^{2+\epsilon}+ B^{4 + \epsilon} + \frac{M}{A}.
\end{equation}

Recall that $B = A^6\frac{T}{M^3}+4\delta A^3$.

Therefore
\begin{equation} \label{S7noB}
|S_7| \ll _\epsilon A^\epsilon  \left ( A^{13}\frac{T^2}{M^6} + A^7\delta^2 + A^{24}\frac{T^4}{M^{12}} + A^{12}\delta^4  \right ) + \frac{M}{A}.
\end{equation}

Minimizing the right-hand-side of the above inequality with respect to $A$, subject to the constraint \eqref{Abound} we obtain 
\begin{equation} \label{S7bound}
\begin{array}{lcl}
|S_7|& \ll _\epsilon &M^\epsilon  \left (M^{1/2}T^{1/7} + M^{12/25}T^{4/25} + M^{7/8}\delta^{1/4} + M^{12/13}\delta^{4/13} \right ) + M\delta + M^{2/7}T^{2/7}+ \\
& &+M^{1/2}T^{1/4}\delta^{1/4}+T/M.
\end{array}
\end{equation}

Combining equations \eqref{S7}, \eqref{S7bound}, and the inequality $|S_6| \ll_\epsilon M^\epsilon |S_7|$ we obtain 
\begin{equation} \label{final}
\begin{array}{lcl}
|S(f,\delta)|& \ll & c(\epsilon)M^\epsilon  \left (M^{1/2}T^{1/7} + M^{12/25}T^{4/25} + M^{7/8}\delta^{1/4} + M^{12/13}\delta^{4/13} + M\delta +  \right .\\
& &\left . +M^{2/7}T^{2/7}+M^{1/2}T^{1/4}\delta^{1/4}+T/M\right ) + M^{1/2}T\delta^{5/2} + M \delta^{1/2} + M(\delta/T)^{1/3}.
\end{array}
\end{equation}

Now, $1 < M^{1/2}T^{1/7}$, $\delta M < \delta M^{1/2}$, $M\delta^{1/2} \ll M^{7/8}\delta^{1/4}$ since $\delta \ll M^{-1/2}$. Also,  $M(\delta/T)^{1/3} \leq M^{2/3}\delta^{1/3} < M^{12/13}\delta^{4/13}$, so four of the terms in \eqref{final} can be deleted.
\end{proof}

\section{The proof of Theorem \ref{newtheta}}

We follow closely the argument in \cite{OTrifonov2002} and add a new ingredient, Theorem \ref{newest}.
\begin{proof}
Without loss of generality we can assume that $8/65 < \theta \leq 0.129$ since the theorem has already been established for $1/8 < \theta < 1/2$ in \cite{HuxTri}. 
We also assume that $x \geq 10^7$. 

Next, we use a lemma from \cite[Lemma 5]{HuxTri}.

\begin{lemma} \label{red}
Let $\epsilon_1 > 0$, $\sqrt{x} \leq \epsilon_1^3 h \leq \epsilon_1^5 x$,  and $h = x^{1/2 + \theta}$. Then
$$Q(x+h) - Q(x) = \frac{\zeta \left (\frac{3}{2}  \right )}{2\zeta(3)} x^\theta + O \left ( \epsilon_1 x^\theta \right ) + O(R_1) + O(R_2)$$
where $R_1$ is the number of pairs of positive integers $(m,k)$ with $\epsilon_1 x^\theta < m \leq x^{1/5}$, $m^2k^3 \in (x,x+h]$, and  
$R_2$ is the number of pairs of positive integers $(m,k)$ with $\epsilon_1 x^\theta < k \leq x^{1/5}$, $m^2k^3 \in (x,x+h]$.
\end{lemma}

We use Lemma \ref{red} with $\epsilon_1 = x^{-0.01}$.

In \cite[pp.316,317]{OTrifonov2002} it has been established that $R_1 \ll x^{0.12}\log^2x$. Since $\frac{8}{65}= 0.12307\ldots > 0.12$, we only need to estimate $R_2$. 

Let $S_2(A,B)$ be the number of integers $k$ in the interval $(A,B]$ for which there exists an integer $m$ with $m^2k^4 \in (x,x+h]$. Then, $R_2=S_2(x^{\theta - 0.01}, x^{1/5})$.

It is easy to check that 
$$|S_2(M,2M)| \leq |\{ k \in (M,2M] \cap {\mathbb Z}\  |\ ||f_1(k)|| < \delta_1 \}|$$
where $f_1(k) = x^{1/2}k^{-3/2}$ and $\delta_1 = x^\theta M^{-3/2}$.

In \cite[pp.317,318]{OTrifonov2002} it was proved that $S_2(x^{\theta - 0.01},x^{0.158}) \ll x^{0.123}$.

The next range is $(x^{0.158}, x^\frac{7}{39}]$. Note that $\frac{7}{39}=0.179\ldots > 0.158$. We split the interval  $(x^{0.158}, x^\frac{7}{39}]$ into dyadic intervals  of the form $(M,2M]$.  We apply Theorem \ref{newest} on the interval $(M,2M]$ with $f_1(k) = x^{1/2}k^{-3/2}$ and $\delta_1 = x^\theta M^{-3/2}$.
We have $T = x^{1/2}M^{-3/2}$ and $\delta_1 = x^{\theta}M^{-3/2}$.

All conditions of Theorem \ref{newest} hold and we obtain
\begin{equation} \label{.17}
\begin{array}{lcl}
S_2(M,2M)&\ll&c(\epsilon) M^\epsilon \left ( x^{2/25}M^{6/25} +   x^{1/14}M^{2/7} +  x^{1/7}M^{-1/7}+ x^{\theta/4}M^{-3/8}  \right . \\
 & &\left .  +x^{4\theta/13}M^{-6/13} \right) + x^{(1+5\theta)/2}M^{-19/4}+x^{(1+2\theta)/8}M^{-1/4}+x^{1/2}M^{-5/2}.
 \end{array}
\end{equation}

Next, we need a lemma which is well-known in  mathematical  folklore. 

Let $T$ be a set of positive integers. For $M \geq 1$,  define $S(M, 2M)$ to be the number of integers in the interval $(M,2M]$ which are elements of $T$. 

\begin{lemma} \label{dyadic}
Let $c_1, \ldots, c_k, d_1, \ldots, d_l, a_1, \ldots, a_k,b_1, \ldots, b_l, u,$ and $v$ be positive constants with $u<v$. Assume that for each $M 
\in (x^u, x^v/2]$  we have 
\begin{equation} \label{SMest}
S(M,2M) \leq c_1M^{a_1} + \cdots + c_kM^{a_k}+d_1M^{-b_1} + \cdots + d_lM^{-b_l}.
\end{equation}
Then, 
$$S(x^u, x^v) \ll c_1x^{a_1 v} + \cdots + c_kx^{a_k v}+d_1x^{-b_1 u} + \cdots + d_lx^{-b_l u}.$$
\end{lemma}

{\it Sketch of a proof.} Split the interval $(x^u, x^v]$ into nonintersecting dyadic intervals. Apply estimate \eqref{SMest} to each dyadic interval and sum the resulting geometric progressions.

Applying Lemma \ref{dyadic} to the estimate \eqref{.17} we obtain $S_2(x^{0.158}, x^{7/39}) \ll_\epsilon x^{8/65 + \epsilon}$. (Among the terms in \eqref{.17} the largest term in this range  is $x^{2/25}M^{6/25}$ when $M$ is near $x^{7/39}$. )

The next range for $M$ is $(x^{7/39}, x^{1/5}]$. Here we use the  estimate \cite[Theorem 2]{HuxTri}.

\begin{lemma} \label{lgM}
Let $f \in {\mathcal F}_2 \cap {\mathcal F}_3$, $1 \leq M \leq T$, $0 \leq \delta \leq \frac{1}{2}C^{-1/2}T^{1/2}M^{-1}$. Then, 
$$\begin{array}{lcl}
|S(f,\delta)| &  \ll & T^{3/10}M^{3/10}(\log M)^{1/2} + T^{4/11}M^{2/11}(\log M)^{5/11}+ \delta^{1/8}T^{3/8}M^{1/4}(\log M)^{5/8}\\
 & &+ \delta^{1/7}T^{1/7}M^{4/7}(\log M)^{5/7}+\delta^{2/5}T^{1/5}M^{3/5}\log M + \delta M.
  \end{array}$$
\end{lemma}

We apply Lemma \ref{lgM} with $f_1(k) = x^{1/2}k^{-3/2}$, $\delta_1 = x^\theta M^{-3/2}$, and $M \in (x^{7/39}, x^{1/5}]$. The conditions of Lemma \ref{lgM} hold and we obtain 
\begin{equation} \label{last}
\begin{array}{lcl}
S_2(M,2M) & \ll & (\log M )\left (  x^{3/20}M^{-3/20} + x^{2/11}M^{-4/11} +  x^{(3+2\theta)/16}M^{-1/2}\right ) \\
 & & +(\log M) \left ( x^{(1+2\theta)/14}M^{1/7} + x^{(1+4\theta)/10}M^{-3/10}+x^{\theta} M^{-1/2} \right ). 
\end{array}
\end{equation}

Combining equation \eqref{last} and Lemma \ref{dyadic} we obtain $S_2\left (x^{7/39},x^{1/5}\right ) \ll (\log M) x^{8/65}$ which completes the proof of the theorem.

\end{proof}

\section{Acknowledgements} The author greatly appreciates the help of D.~R.~Heath-Brown. This paper would not  exist without D.~R.~Heath-Brown who generously shared his new estimate of the number of zeros of ternary quadratic forms in boxes with the author.

\bigskip
\hrule
\bigskip

\noindent 2010 {\it Mathematics Subject Classification}:
Primary 11N25, Secondary 11N37.

\noindent \emph{Keywords: }
squarefull numbers, ternary quadratic forms.

\bigskip
\hrule
\bigskip


\begin{thebibliography}{10}

\bibitem{BatGro}
P.~T.~Bateman and E.~Grosswald, On a theorem of Erd\H{o}s and Szekeres, \emph{Illinois J. Math} \textbf{2} (1958), 88--98.

\bibitem{BombPila}
E.~Bombieri and J.~Pila, On the number of integer points on arcs and ovals, \emph{Duke Math. J.} \textbf{59} (1989), 337--357.

\bibitem{BH2018}
T.~D.~Browning and D.~R.~Heath-Brown, Counting Rational Points on Quadric
Surfaces, \emph{Discrete Analysis},  2018:15.


\bibitem{FilTri}
M.~Filaseta and O.~Trifonov, The distribution of squarefull numbers in short intervals, \emph{Acta Arith.} \textbf{67} (1994), 323--333.

\bibitem{DRHeath-Brown1991}
D.~R.~Heath-Brown, Square-full numbers in short intervals, \emph{Math. Proc. Cambridge Philos. Soc.} \textbf{110}  (1991), 1--3.


\bibitem{HeathBrown2024}
private communication (2024).

\bibitem{Huxley1989}
 M.~N.~Huxley, The integer points close to a curve, \emph{Mathematika} \textbf{36} (1989), 198--215.


\bibitem{HuxSar}
M.~Huxley and P.~Sargos, Points entiers au voisinage d'une courbe plane de classe $C^n$, II, \emph{Functiones et Approximatio} \textbf{XXXV} (2006), 91--115.

\bibitem{HuxTri}
M.~N.~Huxley, The square-full numbers in an interval, \emph{Math. Proc. Cambridge Philos. Soc.} \textbf{119} (1996), 201--208.

\bibitem{IsKe}
E.~Isaacson and H.~B.~Keller, Analysis of Numerical Methods, John Wiley, New York (1966).


\bibitem{CHJia}
C.~H.~Jia, The square-full numbers in short intervals, \emph{Acta Math. Sinica} \textbf{30} (1987), 614--621.


\bibitem{HQLiu}
H.~Q.~Liu, On square-full numbers in short intervals, \emph{Acta Math. Sinica (N.S.)} \textbf{6}  (1990), 148--164.



\bibitem{PShiu1980}
P.~Shiu, On the number of square-full integers between successive squares, \emph{Mathematika} \textbf{27} (1980), 171--178.

\bibitem{PShiu1984}
P.~Shiu, On square-full integers in a short interval, \emph{Glasgow Math. J.} \textbf{25} (1984), 127--134.

\bibitem{PGSchmidt1986}
P.~G.~Schmidt, \"Uber die Anzahl quadratvoller Zahlen in kurzen Intervallen, \emph{Acta Arith.} \textbf{46} (1986), 159--164.

\bibitem{PGSchmidt1988}
P.~G.~Schmidt, Zur Anzahl quadratvoller Zahlen in kurzen Intervallen und ein verwandtes Gitterpunktproblem,  \emph{Acta Arith.} \textbf{50} (1988), 195--201.

\bibitem{Pila}
J.~Pila, Geometric postulation of a smooth function and the number of rational points, \emph{Duke Math. J.} \textbf{63} (1991), 449--463.

\bibitem{SwD1974}
H.~P.~ F.~ Swinnerton-Dyer, The number of lattice points on a convex curve, \emph{J. Number Theory} \textbf{6}
(1974) 128--135.

\bibitem{OTrifonov2002}
O.~Trifonov, Lattice points close to a smooth curve and squarefull numbers in short intervals, \emph{J. London Math. Soc.} (2)  \textbf{65} (2002), 
303--319.



\end{thebibliography}
\end{document}